\documentclass{amsart}
\usepackage{mathptmx, amscd, amssymb, amsmath, enumerate, slashbox, hyperref, float,mathtools, tikz-cd,thmtools, thm-restate}

\DeclareMathAlphabet{\mathcal}{OMS}{cmsy}{m}{n}
\newtheorem{theorem}{Theorem}[section]

\newtheorem{lem}[theorem]{Lemma}
\newtheorem{prop}[theorem]{Proposition}

\theoremstyle{definition}
\newtheorem{definition}[theorem]{Definition}

\numberwithin{equation}{theorem}

\def\phi{\varphi}

\def\to{\longrightarrow}

\def\dim{\operatorname{dim}}
\def\codim{\operatorname{codim}}

\def\rank{\operatorname{rank}}
\def\Hom{\operatorname{Hom}}

\def\Ext{\operatorname{Ext}}

\def\Proj{\operatorname{Proj}}
\def\Sing{\operatorname{Sing}}

\newcommand{\GL}{\mathrm{GL}}

\newcommand{\mf}[1]{\mathfrak{#1}}

\begin{document}
\title{Lengths of Local Cohomology of Thickenings}

\author{Jennifer Kenkel}
\address{Department of Mathematics, University of Kentucky, 719 Patterson Office Tower, Lexington, KY~40507, USA}
\email{jennifer.kenkel@gmail.com}

\thanks{The author was supported by NSF grant DMS~1246989. I thank Claudiu Raicu for his patient explanation at the MSRI, and Anurag Singh for his many helpful discussions.  }

\begin{abstract}
%%%%%%%%%%%%%%%%%%%%%%%%%%%%%%%%%%%%%%%%%%%%%%%%%%%
Let $R$ be a standard graded polynomial ring that is finitely generated over a field of characteristic $0$, let $\mf{m}$ be the homogeneous maximal ideal of $R$, and let $I$ be a homogeneous prime ideal of $R$. Dao and Monta\~{n}o defined an invariant that, in the case that $\Proj(R/I)$ is lci and  for cohomological index less than $\dim(R/I)$, measures the asymptotic growth of lengths of local cohomology modules of thickenings. 
They showed its existence and rationality for certain classes of monomial ideals $I$. 

The following affirms that the invariant exists and is rational for rings $R = \mathbb{C}[X]$ where $X$ is a $2 \times m$ matrix and $I$ is the ideal generated by size two minors and is to our knowledge, the first non-monomial calculation of this invariant.  
\end{abstract}
\maketitle

%%%%%%%%%%%%%%%%%%%%%%%%%%%%%%%%%%%%%%%%%%%%%%%%%%%
\section{Introduction}
\label{section:introduction}
%%%%%%%%%%%%%%%%%%%%%%%%%%%%%%%%%%%%%%%%%%%%%%%%%%%
Let $R$ be a standard graded polynomial ring that is finitely generated over a field, let $\mf{m}$ be the homogeneous maximal ideal of $R$, and let $I$ be a homogeneous prime ideal of $R$. Let $d$ be the dimension of $R$.

 The invariant \emph{multiplicity} (see \cite[Definition 4.1.5]{brunsAndHerzog}) is defined for $\mf{m}$-primary ideals as
$$ e(I) \coloneqq d! \lim\limits_{t \to \infty} \frac{\ell(R/I^t)}{t^{d}}.$$

In \cite{Katz},  the authors defined a generalization of this invariant, with which $e(I)$ agrees when $I$ is $\mf{m}$-primary, 
$$\epsilon^0(I) \coloneqq d! \lim \frac{\ell \left( H^0_{\mf{m}}(R/I^t) \right) }{t^d}.$$
They showed this invariant gives information about the analytic spread of $I$. 

Dao and Monta\~{n}o defined a further generalization of multiplicity in \cite{DaoAndMontano}:   
$$ \epsilon^j(I) := \lim_{t \to \infty} \dfrac{\ell \left( H^j_{\mathfrak{m}} (R/I^t) \right) }{t^{d}}.$$  

\par 
This numerical invariant also builds on ideas of Bhatt, Blickle, Lyubeznik, Singh, and Zhang.  
They showed, in \cite[Theorem 1.1]{BBLSZ}, that when $R/I$ is a complete intersection on the punctured spectrum, then for $k < \dim\Sing(R/I)$, the induced maps on graded components, $H^k_{\mf{m}}(R/I^t)_d \to H^k_{\mf{m}}(R/I^{t-1})_d$, are isomorphisms for $t$ sufficiently large. 
The invariant $\epsilon^j$, on the other hand, measures the behavior of local cohomology modules in every graded component at once. 
Dao and Monta\~{n}o related $\epsilon^j$ to depth conditions on the Rees algebra, and computed $\epsilon^j(I)$ for monomial ideals. Furthermore, in \cite[Corollary 5.4]{DaoAndMontano},
they showed that, when $R=\mathbb{C}[x_{i,j}]$ a polynomial ring in $pq$ variables, and $I$ a $\GL$-invariant ideal that is a thickening of a determinantal ideal, then for every $i \leq p+q-2$ 
$$\limsup_{t \to \infty} \dfrac{\ell \left( H^j_{\mathfrak{m}} (R/I^t) \right) }{t^{d}} < \infty.$$ 
They also show in \cite[Theorem 6.4]{DaoAndMontano}, that in the case that $\Proj(R/I)$ is lci, for every $i< \codim(\Sing(R/I)))$ such that $H^{d-i}_I(R) \neq 0$, we have 
$$\liminf_{t \to \infty} \dfrac{\ell \left( H^j_{\mathfrak{m}} (R/I^t) \right) }{t^{d}} > 0.$$
A natural route for study, then, is to determine whether the invariant $\epsilon^j(I)$ exists.
The following affirms the limit $\epsilon^j(I)$ does exist in the particular family $R=\mathbb{C}[X]$, where $X$ a $2 \times m$ matrix of variables, and $I$ the ideal generated by all size two minors. To our knowledge, this is the first calculation of $\epsilon(I)$ where $I$ is not a monomial ideal. 

New veins of inquiry in commutative algebra are often first mined in the family of determinantal rings, as they offer enough algebraic complexity to be nontrivial, but can be appraised through techniques from other fields of study (such as representation theory, which this paper uses liberally) \cite{EagonNorthcott, HochsterAndEagon, DEP}
. For a detailed study of determinantal rings, see \cite{DetRings}.

\par Let $X = ( x_{ij} )$ be a $2 \times m$ matrix of indeterminates, with $2<m$, and let $R = \mathbb{F}[X]$, where $\mathbb{F}$ is a field of characteristic zero.  The goal of this article is to determine the length, $\ell(H^j_{\mf{m}}(R/I^t, R))$. 
Following the notation of Raicu, Weyman, and Witt, this paper will refer to the number of rows of $X$ as $n$ and the number of columns as $m$. 
However, differing from their notation, we will use $t$ (rather than $d$) to denote the exponent of an ideal. When repeating others results, we will use $D$ (rather than $p$) to denote the size of the minors, but in this paper, we will always specialize to $D$ equal to two.

Our main results are
\begin{restatable}{theorem}{myThm} For $R= \mathbb{C}[X]$ where $X$ a $2 \times m$ matrix, and $I$ the ideal generated by size two minors, then 
	$$ \ell(H^3_{\mf{m}} (R/I^t)) = \frac{1}{m+1}{m + t -2 \choose m}{m + t - 1 \choose m}.$$ 
\end{restatable}  
And its corollary, 
\begin{restatable}{corollary}{myCor} With the above conditions, 
	$$\epsilon^3(I) = \lim\limits_{t\to\infty} \frac{\ell(H^3_{\mf{m}} R/I^t)}{t^{d}} = \frac{1}{(m+1)(m!)(m!)} $$ 
	and 
	\begin{align*}
	d! \epsilon^3 (I) &=  
	d!  \lim\limits_{t\to\infty} \frac{\ell(H^3_{\mf{m}} R/I^t)}{t^{d}} \\ 
	&=\frac{(2m)!}{(m+1)(m!)(m!)} \\
	&= \frac{1}{m+1}{2m \choose m}
	\end{align*} 
	Interestingly, $	d! \epsilon^3(I)$ is the $m^{th}$ Catalan number. 
\end{restatable} 

The Catalan numbers are perhaps the most well-studied sequence in combinatorics. The reason for this association is still unclear, but hints at an elegant interpretation of $\epsilon^k(I)$.

 \section{Justification of Cohomological Index}  \label{finiteLengthsSection}
\par  In later sections, we will calculate the lengths of all local cohomology modules, $H^3_{\mf{m}}(R/I^t)$ whenever $R$ is a field adjoin a $2 \times m$ matrix of variables and $I$ is the ideal generated by size two minors. 
In this section, we show that the choice of cohomological index three is not arbitrary, but in fact the only case in which $\epsilon^j(I)$ exists and is nontrivial. We show that $H^j_{\mf{m}}(R/I^t)$ is zero at all indices other than $j=3$ or $j=\dim(R/I^t)$, and recall that the lengths of top local cohomology modules are infinite. 
\par We note that we could calculate the lengths of local cohomology by calculating instead the lengths of related $\Ext$-modules. 
The ring $R = \mathbb{F}[X]$ is regular, and so by graded duality, 
$$ H^j_{\mf{m}}(R/I^t)) \cong \Ext_R^{mn-j}(R/I^t, R(-a))^{\vee}.$$
Thus the length, $\ell(H^j_{\mf{m}}(R/I^t))$, is equal to $\ell(\Ext_R^{mn-j}(R/I^t, R))$.
\par Making use of graded duality, we will show all other local cohomology modules are zero by showing that, for all $t$, the modules $\Ext^i_R(R/I^t, R)$, are $0$  whenever $i$ differs from $2m-3$ or $m-1$.  In order to do this, we will show that the maps $\Ext^i_R(R/I^t, R)$ to $\Ext^i_R(R/I^{t+1}, R)$ are injective, and then we will recall a result that the direct limits $\lim\limits_{t\to \infty}\Ext^i_R(R/I^t, R)$ are 0 for all $i \neq 2m-3$ or $m-1$. 
%Then we will recall that the module $\Ext^{m-1}_R(R/I^t, R)$ has infinite length, because as it is dual to the top local cohomology module of $R/I^t$. 
The following result is adapted from \cite[Equation 4.8]{RWW}. 
\begin{lem} \label{extInjLemma} In the case that $R=\mathbb{F}[X]$ with $\mathbb{F}$ a field of characteristic 0, and $I$ the ideal generated by maximal minors, i.e., $d=n$, the natural maps from $\Ext^j_R(R/I^t, R)$ to $\Ext^j_R(R/I^{t+1}, R)$ are injective for all $t$ and all $j$. 
\end{lem}  
\begin{proof} 
	The successive maps 
	$$ \Ext^j_R(I^t, R) \to \Ext^j_R(I^{t+1}, R)$$
	are injective for all $t$ and all $j$ by \cite[Equation 4.8]{RWW}. 
	The short exact sequence 
	$$ 0 \to I^{t} \to R \to R/I^{t} \to 0 $$ 
	induces the long exact sequence 
	$$ \cdots \to \Ext^{j}_R(R, R) \to \Ext^j_R(I^t, R) \to \Ext^{j+1}_R(R/I^t, R) \to \Ext^{j+1}(R, R) \to \cdots .$$ 
	Since $\Ext^j_R(R, R) = 0$ for all $j \neq 0$,  the above long exact sequence gives $$\Ext^{j}_R(I^t, R) \cong \Ext^{j+1}_R(R/I^t, R)$$ for all $j \geq 1$. As  $\Ext^0_R(R/I^t, R) = \Hom(R/I^t, R)$ and $\Hom(R/I^t, R)$ is 0, we have the maps from  $\Ext^0_R(R/I^{t}, R)$ to $\Ext^0_R(R/I^{t+1},R)$ are injective. 
	It remains to show that the natural maps from $\Ext^1_R(R/I^t, R)$ to $\Ext^1_R(R/I^t, R)$ are injective. 
	\par  As $\Ext^0(I^t, R) = \Hom(I^t, R)$, and as $I^{t+1}$ is a subset of $I^t$, the natural maps from $\Ext^0_R(I^t, R)$ to  $\Ext^0(I^{t+1}, R)$ are restriction maps, and thus are injective. We have the following diagram in which the rows are exact: 
	\\
	\begin{center} 
		\begin{tikzcd}
		&\Ext^{0}_R(R,R) \arrow[r] \arrow[d,equal] &\Ext^0_R(I^t, R) \arrow[r] \arrow[d,hookrightarrow] & \Ext_R^1(R/I^t, R) \arrow[r] \arrow[d] &0 \arrow[d,equal] \\
		&\Ext^{0}_R(R,R) \arrow[r] &\Ext^0_R(I^{t+1}, R) \arrow[r] & \Ext^1_R(R/I^{t+1}, R) \arrow[r] &0 \\
		\end{tikzcd}
	\end{center} 	
	From the diagram, the maps from $\Ext^1_R(R/I^t, R)$ to $\Ext^1_R(R/I^{t+1}, R)$ are injective by an application of the four lemma, or a straightforward diagram chase. Thus we have that the natural maps from $\Ext^j_R(R/I^t, R)$ to $\Ext^j_R(R/I^{t+1}, R)$ are injective for all $j$ and for all $t$. 
\end{proof} 

We will now use Lemma~\ref{extInjLemma} to show that the modules of interest, $\Ext^j_R(R/I^t, R)$, vanish except at two cohomological indices. We will also show that, at those two cohomological indices, the modules $\Ext^j_R(R/I^t, R)$ do not vanish for some $t$ sufficiently large. 
\par The next proposition follows from the above lemma and \cite[Theorem V.10, part~(b)]{Witt}
\begin{prop} \cite{Witt, RWW} If $X$ is a $2 \times m$ matrix and $I$ is the ideal generated by all size two minors, then for all $t$, the modules $\Ext^i_R(R/I^t, R)$ are 0  whenever $i$ differs from $2m-3$ or $m-1$. For $t \gg 0$, the modules $\Ext^i_R(R/I^t, R)$ are not 0 when $i = 2m-3 \text{ or } m-1 $.  
\end{prop} 
\begin{proof} 
	\par When $R = \mathbb{F}[X]$ with $X$ a $2 \times m$ matrix, and $I$ is the ideal generated by all size $2$ determinants of $X$,
	$$ \dim(R/I) =  m + 1.$$
	\par  For maximal minors, that is, $d=n$, \cite{Witt}  gives that $H_I^j(R)$ is non-vanishing if and only if the cohomological index, $j$, is equal to $(n-r)(m-n) + 1$ for some $0 \leq r < n$. Specializing to $n=2$,  we then have that $r$ can only take the values $0$ or $1$. Therefore, the module $H_I^j(R)$ is nonzero precisely at cohomological indices $2m - 3$ and $m-1$. 
	\par Recall that
	$$ H^j_I(R) = \lim\limits_{t \to \infty} \Ext^j_R(R/I^{t}, R).$$
	This formulation of $H^j_I(R)$, together with Witt's result on non-vanishing local cohomology modules, gives that $\Ext^j(R/I^t, R)$ is not $0$ for $t \gg 0$ whenever $j = 2m-3$ or $m-1$.  
	On the other hand, when $j\neq 2m-3$ or $m-1$, the fact that
	\begin{equation*} 
	\lim\limits_{t\to\infty} \Ext^j_R(R/I^{t+1}, R) = 0.
	\end{equation*} 
	together with lemma~\ref{extInjLemma} imply that $\Ext^j_R(R/I^t, R) = 0$ for all $t$. 
\end{proof} 

\par We have shown that the modules $\Ext^{j}_R(R/I^t, R)$ are non-zero for exactly two values of $j$; that is, $j= m-1$ and $j=2m-3$. Note that, in the case that $j=m-1$,
\begin{equation}
 \Ext^{m-1}(R/I^t, R)^{\vee} \cong H^{m+1}_{\mf{m}}(R/I^t)
\end{equation} 
and that $\dim(R/I^t)=m+1$. Since the local cohomology module of a Noetherian local ring at cohomological index $\dim(R/I)$ has infinite length, we have determined that the only cohomological index for which $\Ext^j_R(R/I^t, R)$ has finite, non-zero length is $j=3$. 
Since $\Ext^{2m-3}_R(R/I^t, R)$ is not zero for $t$ sufficiently large, by graded duality, the module $H^3_{\mf{m}}(R/I^t)$ must also be nonzero for $t$ sufficiently large. Now that we have determined that $H^j_{\mf{m}}(R/I^t)$ has nonzero finite length exactly when $j$ is equal to 3.

\section{Length Calculation} \label{lengthsCalculationSection}
In this section, we will calculate the lenght of $\Ext^{2m-3}_R(R/I^t, R)$ by using representation theoretic techniques to break the module down into a sum of $\Ext$ modules, and then further breaking those $\Ext$ modules into vector spaces. 

\par We rely heavily on Theorem 3.2 of \cite{Raicu}, stated in full in Theorem \ref{extSplitThm}. The details of this argument have been suppressed. However, the broad strokes of the argument are as follows. Given some module $M$ with a finite graded filtration, 
$$M_{\bullet}: 0 = M_0 \subset M_1 \subset \dots \subset M_r = M, $$ 
the module $\Ext^{j}_R(M,R)$ embeds as a graded subspace of $\bigoplus^{r-1}_{i=1}\Ext_R^j(M_{i+1}/M_{i},R)$ and this embedding can be chosen to be $\GL$-equivariant. When this map is not just an embedding but a vector space isomorphism, i.e.,  $$\Ext^{j}_R(M,R) \cong \bigoplus_{i=1}^{r-1}\Ext_R^j(M_{i+1}/M_{i},R),$$ the filtration is called \textbf{$\Ext$-split}.
\par We will use an $\Ext$-split filtration to break up $\Ext^j_R(R/I^t, R)$ into a direct sum of modules, $Ext^j_R(J_{\underline{z},l},R)$. We will further break \emph{those} modules up into a direct sum of modules described by Schur functors acting on vector spaces over $\mathbb{C}$, the lengths of which are straightforward to calculate. This decomposition is an isomorphism of graded vector spaces over $\mathbb{C}$ not necessarily as $R$-modules, but for the purpose of calculating lengths, it suffices. 

\subsection{Describing the Ext-split filtration}
\par In order to determine $J_{\underline{z},l}$, we determine the subset $\mathcal{X}$ of partitions to which $I$, the ideal generated by size two minors, is associated. Let $P_n$ denote partitions into $n$ components, in decreasing order. Recall from \cite{DEP} that we can associate to any GL-invariant ideal a subset of partitions, $\mathcal{X} \subseteq P_n$. To define the subset of partitions associated to $t^{th}$ powers of the ideal of size $d$ determinants, $I^t_{d}$, we make the following definitions.
\begin{definition} We define for each $l=1, \dots, n$ the polynomial 
	$$ {\det}_l = \det\left(X_{ij}\right)_{1\leq i, j \leq l} .$$ \end{definition} 
In other words, for some integer $l$, the polynomial $\det_l$ is the determinant of the $l \times l$ submatrix that is in top left corner of $X$. We will use the notion of a determinant with respect to one integer to define the determinant function with respect to a partition. But in order to do that, we must first define the conjugate of a partition.  
\begin{definition} For a partition $\underline{x}$ we define the conjugate partition, $\underline{x}'$ which swaps the roles of rows and columns of the Young diagram associated to $\underline{x}$. That is, $x_i'$ counts the number of boxes in column $i$ of the young diagram associated to the partition $\underline{x}$. \end{definition} 
For example, if $\underline{x} = (5,3,2)$, then $\underline{x}' = (3, 3, 2, 1, 1)$.  
\begin{definition}
	For $\underline{x} \in \mathcal{P}_n$, we define
	$$ {\det}_{\underline{x}}=\prod_{i=1}^{x_1}{\det}_{x_i'}.$$ 
\end{definition} 
The polynomial $\det_{\underline{x}}$ will be a product of minors of $X$, possibly minors of different sizes. 

\begin{definition}  Let $I_{\underline{x}}$ be the GL-orbit of $\det_{\underline{x}}$. \end{definition}  
And finally, 
\begin{definition}
	If $\mathcal{X} \subset P_n$, then define $I_{\mathcal{X}}=\sum\limits_{\underline{x}\in \mathcal{X}} I_{x}.$
\end{definition}
According to \cite{DEP}, the $t^{th}$ power of the ideal of size $D$ determinants, $I_D^t$, corresponds to  the ideal $I_{\mathcal{X}}$ given by the subset of partitions $\mathcal{X}$, where $\mathcal{X}$ is given by:  
$$\mathcal{X} = \{ \underline{x} \in \mathcal{P}_n : \lvert x \rvert = tD, x_1 \leq t \} .$$ 
Specializing to the case that $D=n=2$,  gives
$$ \mathcal{X}^t_{n} =  \{ \underline{x} \in \mathcal{P}_n : \lvert x \rvert = 2t, x_1 \leq t \}.$$ 
that is, the set $\mathcal{X}$ contains exactly the partition: $(t,t)$. 
\par In other words, the $t^{th}$ power of the ideal generated by size two minors, $I^t_2$, is equal to $I_{\mathcal{X}}$, where $\mathcal{X}$ contains exactly one partition. Since $\mathcal{X}^t_2$ contains only one element, $I_{\mathcal{X}}$ is equal to the ideal $I_{(t,t)}$, which is in turn, the $\GL$-orbit of $\det_{(t,t)}$. The polynomial $\det_{(t,t)}$ is the product 
\begin{equation} 
\prod_{i=1}^t {\det}_{2}.
\end{equation} 
We record this information less for the purpose of applying Theorem \ref{extSplitThm}. 

\par We have a system of comparing $\GL$-invariant ideals.  If $\underline{y}$ a partition, we write $\underline{x}~\subset~\underline{y}$ to indicate $x_i \leq y_i$ for all $i$, or equivalently,
$I_{\underline{x}}\subseteq I_{\underline{z}}$. A partition is the same with or without trailing $0$'s, e.g.,
the partition $(3,2,1)$ is the same as the partition $(3,2,1,0,0,0)$. By appending $0$'s, any two partitions can be compared. 
%We often omit the trailing 0's... 

\begin{definition} 
	The collection of partitions $\mathfrak{succ}(\underline{z},l)$, given some positive integer $l$ is defined to be 
	$$\mathfrak{succ}(\underline{z},l)=\{ \underline{x}\in \mathcal{P}_n :\underline{x}>\underline{z} \text{ and } x_i > z_i \text{ for some } i>l \}. $$
\end{definition} 
Note that for all $\underline{x}\in \mathfrak{succ}(\underline{z}, l)$, one has inclusion of the corresponding ideals, $I_{\underline{x}} \subseteq I_{\underline{z}}$, so the following definition makes sense:

\begin{definition} $$J_{\underline{z}, l} \coloneqq I_{\underline{z}, l} / I_{\mathfrak{succ}(\underline{z}, l)}$$ \end{definition} 

We are now ready to state theorem \ref{extSplitThm}, which will allow us to break $\Ext^{2m-3}_R(R/I^t, R)$ into a direct sum of other $\Ext$ modules. 
\begin{theorem}[\cite{Raicu17}] \label{extSplitThm}
	Let $\mathcal{X} \subseteq P_n$ and let $I_{\mathcal{X}} \subset R$ denote the associated $\GL$-invariant ideal. There exists a $\GL$-invariant $\Ext$-split filtration of $R/I_{\mathcal{X}}$ whose factors are the modules $J_{(\underline{z}, l)}$ for $(\underline{z}, l) \in \mathcal{Z}(\mathcal{X})$ and therefore we have for each $j \geq 0$ a $\GL$-equivariant isomorphism of graded vector spaces $$ \Ext^j_R(R/I_{\mathcal{X}},R) \cong \bigoplus\limits_{(\underline{z},l) \in \mathcal{Z}(\mathcal{X})} \Ext^j_R(J_{\underline{z},l}, R),$$
	where the set $\mathcal{Z}(\mathcal{X})$ is given by  \begin{equation}
	\label{zset} 
	\begin{split} 
	\mathcal{Z}(\mathcal{X}^t_D) =
	\{
	&(\underline{z}, l): 0 \leq l \leq D-1, \ \underline{z} \in \mathcal{P}_n, \  z_1 = \dots = z_{l+1} \leq t-1 \text{ and } \\ 
	& \lvert \underline{z} \rvert + (t-z_1)  l + 1 \leq D  t \leq \lvert \underline{z} \rvert + (t-z_1)(l+1) 
	\}
	\end{split}  
	\end{equation}

\end{theorem}

\par The set $\mathcal{Z}(\mathcal{X}^t_D)$ is a collection of partitions of length $2$, $\underline{z} = (z_1, z_2) $ and an integer, $l$. In the following two lemmas, we completely characterize the set $\mathcal{Z}(\mathcal{X}^t_D)$. 
\\ \begin{lem} In the above decomposition, the integer $l=1$ for all $(\underline{z}, l) \in \mathcal{Z}(\mathcal{X}^t_D)$.  
\end{lem} 
\begin{proof} The very first inequality gives that $0 \leq l \leq 1$, that is, $l$ is either $0$ or $1$. Assume that $l=0$, and specialize to the case that $D=2$. By the second line of (\ref{zset}), we have
	\begin{alignat*}{3}
	z_1 + z_2 + (t-z_1)\cdot 0 + 1 & \leq Dt && \leq z_1 + z_2 + (t - z_1)\cdot (0+1) \\
	z_1 + z_2 + 1 & \leq 2t && \leq z_1 + z_2 + t - z_1 \\ 
	z_1 + z_2 + 1 &\leq 2t && \leq z_2 + t \\ 
	z_1 + 1 &\leq 2t - z_2 && \leq t   
	\end{alignat*}
	Specifically, $2t-z_2 \leq t$, so $z_2 \geq t$. However, by the top line of Equation (\ref{zset}), $z_1 \leq t-1$. Since $\underline{z}$ is a partition, $z_1 > z_2$. Thus, it cannot be that $l= 0$, and so $l=1$. 
\end{proof} 
\begin{lem} 

	The set $\mathcal{Z} \left( \mathcal{X}^t_D \right)$ is exactly 
	$$ \{ ((z, z),1), z \leq t-1 \}$$ 
\end{lem} 
\begin{proof} 

	\par Given that $l=1$, the top line of Equation (\ref{zset}) gives that $z_1 = z_2 $. From here on, we will denote $z_1 = z_2$ as $z$. The top line also gives that $z \leq t-1$. 
	\par The second line of Equation (\ref{zset}) gives no further restrictions. The second line requires that 
	\begin{alignat*}{3}
	z+z + (t-z)\cdot 1 + 1 &\leq 2t &&\leq z + z + (t-z)\cdot 2 \\ 
	2z + t - z + 1 & \leq 2t &&\leq 2z + 2(t - z) \\ 
	z + t + 1 &\leq 2t &&\leq 2t  
	\end{alignat*} 
	That is, the condition of the second line is simply $z+1 \leq t$, which is the exact same as the condition $z \leq t-1$ given by the first line. 
	\par Therefore, in the case of $n=p=2$, 
	$$\mathcal{X}^t_p = \{ ((z, z),1) : z \leq t-1  \}. \qedhere 
	$$
\end{proof}  
\par Thus, for each successive power of $I^t$, we have exactly one $\underline{z}$ that was not a component of $I^{t-1}$:  $$\underline{z} = (t-1, t-1).$$
\subsection{Decomposing $Ext^j_S(J_{\underline{z},l}, S)$}
% \begin{center} Determining $\lambda, s, t_1, \lambda(s)$ \end{center}
\par Breaking $\Ext^j_S(S/I^t_d, S)$ into components $\Ext^j_S(J_{\underline{z}, l}, S)$ is helpful because \emph{these} modules have been completely described. To understand this description, we first must define dominant weights and their associated Schur functors. 
\begin{definition} A dominant weight is a $\lambda \in \mathbb{Z}^{\mathbb{N}}$ such that $\lambda_1 \geq \lambda_2 \geq \dots \geq \lambda_N$. The associated Schur functor has the property that $$ \mathbb{S}_{\lambda + (1^N) }\mathbb{C}^N \cong \mathbb{S}_{\lambda } \mathbb{C}^N \otimes \bigwedge^{N}\mathbb{C}^N $$
\end{definition} 

Using theorem \ref{extSplitThm}, we broke the module of interest, $\Ext^{2m-3}(R/I^t, R)$, into a direct sum of modules of the form $\Ext^{2m-3}(J_{\underline{z}, l}, R)$. We now give a vector space isomorphism of modules of the form $\Ext^j_{R}(J_{(\underline{z},l)}, R)$ and direct sums of Schur functors acting on vector spaces over $\mathbb{C}$. The benefit of this is that the dimension of the image of a vector space under the action of a Schur functor is straightforward to calculate.  
\begin{theorem}[\cite{RW14, Raicu}]
	$$\Ext^j_S(J_{\underline{z}, l}, R) = \bigoplus\limits_{\substack{0 \leq s \leq t_1 \leq l\\ m\cdot n - l^2 - s(m-n) -2 \sum t_i = j \\ \lambda \in W(\underline{z}, l; \underline{t}, s)}} \mathbb{S}_{\lambda(s)}\mathbb{C}^m \otimes \mathbb{S}_{\lambda} \mathbb{C}^n $$
	where $  W(\underline{z}, l; \underline{t}, s)$ is the set of dominant weights $\lambda \in \mathbb{Z}_{dom}^2$ satisfying:  
	$$
	\begin{cases} 
	\lambda_n \geq l - z_l -m \\ 
	\lambda_{t_i + i } = t_i - z_{n+1-i} - m \text{ for } i = 1, \dots, n-l \\ 
	\lambda_s \geq s-n \text{ and } \lambda_{s+1} \leq s-m 
	\end{cases} 
	$$
	and $\mathbb{S}_{\lambda}$ denotes the Schur functor associated to the partition or dominant weight $\lambda$. \end{theorem}   
\begin{lem} In the above set $W(\underline{z}, l; \underline{t}, s)$ the integer $t_1=1$ and the integer $s=0$. \end{lem}  
\begin{proof} 
	From above, we know that $l=1$. Thus we have that $s, t_1 \in \{ 0,1\}$. The second line under the direct sum demands that $2m - 1 -s(m-n)-2t_1 = j$. 

	 So, in this case,$$mn - l^2 - s(m-n) -2 \sum t_i = j$$ turns into: 
	\begin{align*}
	2m - 1 - s(m-2) -2t_1 &= 2m-3 \\ 
	2m - 1 - sm + 2s -2t_1 &= 2m-3  
	\end{align*} 
	If $t_1 = 0$, then $s=0$ as well, since $s \leq t$. So the equation becomes: \begin{align*} 
	2m - 1 = 2m -3 
	\end{align*} 
	a contradiction.  Thus $t_1 = 1$. 
	\\ Now the equation becomes 
	\begin{align*}
	2m - 1 - s(m-2) -2 &= 2m-3 \\ 
	2m - 3 - s(m-2) &= 2m - 3 \\
	s(m-2) &= 0
	\end{align*}
	Since we assume that $m>2$, it must be that $s=0$.  
\end{proof} 
\begin{lem} The set $W(\underline{z}, 1; (1), 0)$ consists of all $\lambda = (\lambda_1, \lambda_2 )$ of the following form: 
	\begin{align*}
	1-z-m \leq &\lambda_1 \leq -m   \\
	&\lambda_2 = 1-z-m 
	\end{align*} 
\end{lem} 

\begin{proof}  
	Substituting in $n=2, l =1, t_1 = 1$ and $s= 0$ in the description of $W(\underline{Z}, 1; (1),0)$ gives: 
	$$
	\begin{cases} 
	\lambda_2 \geq 1 - z_1 -m \\ 
	\lambda_{2} = 1 - z_{2} - m \text{ for } i = 1\\ 
	\lambda_{1} \leq -m 
	\end{cases} 
	$$ Thus, we have, for fixed $z$, the integer $\lambda_2$ is defined to be $-1 - z - m$. 
	The integer $\lambda_1$ satisfies $\lambda_1 \leq -m$, and since $\lambda$ is a dominant weight, $\lambda_1 \geq \lambda_2$. We have entirely specified what form $\lambda$ takes.
\end{proof} 
\subsection{Calculating Lengths} 
We will calculate the rank of components of $\Ext^{2m-3}_R(R/I^{t}, R)$ that do not appear in $\Ext^{2m-3}_R(R/I^{t-1}, R)$ by considering only when $z=t-1$.

\par A straightforward argument on induced long exact sequences on $Ext$, combined with the vanishing of $Ext$ modules at other indices from \cite{RWW}, shows that, by calculating the rank of components in $\Ext^{2m-3}_R(R/I^t, R)$ that do not appear in $\Ext^{2m-3}_R(R/I^{t-1}, R)$, we are in fact calculating $\ell (\Ext^{2m-3}_R (I^{t-1}/I^{t}),R)$. We mainly use this fact for notational ease.
\par In the case that $z=t-1$, the partition $\lambda$ is $(\lambda_1, 2 - t - m)$ where $2-t-m \leq \lambda_1 \leq -m$. We write $\lambda_1$ as $$\lambda_1 = 2-t-m + \epsilon$$ where $\epsilon \in \{0, \dots, t-2 \}$ .  
Recall that we can think of $\Ext^{2m-3}_R(R/I^t, R)$ as direct sums of objects of the form: 
$S_{\lambda(s)}\mathbb{C}^m \otimes S_{\lambda} \mathbb{C}^2 $. So in order to calculate $\ell(\Ext^{2m-3}_R(R/I^t, R))$, we must calculate $\rank(\mathbb{S}_{\lambda}\mathbb{C}^N)$. 

\begin{lem} \label{Schur func}   $\rank \left(\mathbb{S}_{\lambda}\mathbb{C}^N\right) = \rank\left( \mathbb{S}_{\lambda+(c^N)}\mathbb{C}^N \right)$
\end{lem}

\begin{proof} In the proof of Theorem 3.2, \cite{Raicu17} gives
	$$ \mathbb{S}_{\lambda + (1^N) } \cong \mathbb{S}^{\lambda } \mathbb{C}^N \otimes \bigwedge^{N}\mathbb{C}^N $$ 
	Since $\bigwedge^N \mathbb{C}^N$ has rank $1$, 
	$$ \rank\mathbb{S}_{\lambda + (1^N) } = \rank\mathbb{S}^{\lambda } \mathbb{C}^N $$
	and by induction, 
	$$ \rank\mathbb{S}_{\lambda + (c^N) } = \rank\mathbb{S}^{\lambda } \mathbb{C}^N .$$
\end{proof} 
Then 
\begin{align*} 
\rank(\mathbb{S}_{\lambda} \mathbb{C}^2) &= \rank(\mathbb{S}_{(-2-t - m + \epsilon, -2-t-m)}\mathbb{C}^2) \\
&= \rank(\mathbb{S}_{(-2-t - m + \epsilon +(2+t+m), -2-t-m+(2+t+m))}\mathbb{C}^2)\\ 
&= \rank(\mathbb{S}_{(\epsilon, 0)}\mathbb{C}^2)
\end{align*}
The Schur functor $\mathbb{S}_{(N)}\mathbb{C}^2$ is the vector space of $N^{th}$ symmetric powers on $\mathbb{C}^2$, so the rank of $\mathbb{S}_{(\epsilon,0)}\mathbb{C}^2$ is $\epsilon + 1$.

\par It remains to calculate rank $\mathbb{S}_{\lambda(0)}\mathbb{C}^m$.  The dimension of a vector space acted on by a Schur functor is given in \cite{fultonandharris}: 
$$ \dim \mathbb{S}_{\lambda} V = \prod_{1 \leq i < j \leq m} \dfrac{\lambda_i - \lambda_j + j - i}{j-i} $$
%\textbf{ Note to self: $\prod_{1\leq i < j \leq k} j- i = \prod_{l=1}^{k} l! $}
\\ We also have $\lambda(s) = (\lambda_1, \dots, \lambda_s, (s-2)^{m-2}, \lambda_{s+1} + (m-2),\dots, \lambda_2 + (m-2)) \in \mathbb{Z}_{dom}^m $
(where the exponent $m-2$ denotes that term is repeated $m-2$ times). Since $s=0$, we have: 
\begin{align*} \lambda(0) &= (-2, -2, \dots, -2, \lambda_1 + m - 2, \lambda_2 + m - 2) \\
&= (-2, -2, \dots, -2, -t+\epsilon, -t) 
\end{align*} 
By lemma \ref{Schur func}, $\rank \mathbb{S}_{(-2, -2, \dots, -2, -t+\epsilon, -t)} \mathbb{C}^m = \rank \mathbb{S}_{(t-2, t-2, \dots, t-2, \epsilon, 0)} \mathbb{C}^m $. 

\begin{lem} The dimension of $\mathbb{S}_{(t-2, t-2, \dots, t-2, \epsilon, 0)}\mathbb{C}^m$ is
	$$ \frac{(\epsilon+1)}{m-1}{m+t-3 \choose t-1}{m+t-4-\epsilon \choose m-2 }$$ \end{lem} 
\begin{proof} 
	First note that $\lambda_i - \lambda_j = 0 $ for all $j\leq m-2$. The dimension of a Schur functor acting on a space is given by: 
	\begin{equation*} 
	\dim \mathbb{S}_{((t-2)^{m-2},\epsilon,0)} \mathbb{C}^m = \left( \prod_{1\leq i < j \leq m-2 }  \frac{j-i}{j-i} \right) \left( \prod_{1 \leq i < j, j\in \{m-1, m \}  }\frac{\lambda_i - \lambda_j + j - i }{j-i} \right)
	\end{equation*} 
	(see \cite{fultonandharris}), which we can rewrite into
	\begin{align*}
	&\left( \prod_{1 \leq i < j, j\in \{m-1, m \}  }\frac{\lambda_i - \lambda_j + j - i }{j-i} \right) \\ 
	= 	&\left( \prod_{1 \leq i < m-1  }\frac{\lambda_i - \lambda_{m-1} + {m-1} - i }{m-1-i} \right) \cdot 	\left( \prod_{1 \leq i < m }\frac{\lambda_i - \lambda_m + m - i }{m-i} \right) \\ 
	\end{align*} 
	It is a straightforward series of calculations to find that the term of the above product that the first term in the above product, corresponding to $j=m-1$, is equal to 
	$$ { m+t - \epsilon - 4  \choose m-2 } $$
	while the product corresponding to $j=m$ is equal to 
	$$\left( \epsilon + 1 \right){ m+t-3 \choose m-2 }  \frac{1}{(m-1)}. $$

\end{proof} 
Since $\dim(S_{\lambda_1, \lambda_2} \mathbb{C}^2) = \dim(S_{\lambda} \mathbb{C}^2) = \epsilon+1$, 
$$ \dim(S_{\lambda} \mathbb{C}^2 \otimes S_{\lambda} \mathbb{C}^m )  = \dfrac{(\epsilon + 1)^2}{m-1}{m+t-3 \choose m-2 }{m+t-4-\epsilon \choose t-\epsilon - 2 }$$ 

The above calculation was done for a particular $\epsilon$. To find the total rank of the module $\Ext^{2t-3}_{R}(I^{t-1}/I^t, R)$, we must sum up over $\epsilon$: 
\begin{align*}  &\sum\limits_{\epsilon=0}^{t-2} \dfrac{(\epsilon + 1)^2}{m-1}{m+t-3 \choose m-2 }{m+t-4-\epsilon \choose t-\epsilon - 2 }\\ 
&= \dfrac{1}{m-1}{m+t-3 \choose m-2 }\sum\limits_{\epsilon=0}^{t-2} (\epsilon+1)^2 {m+t-4-\epsilon \choose t-\epsilon - 2 } 
\end{align*} 

\begin{lem} The following is a combinatorial identity:
	$$\sum\limits_{\epsilon=1}^{b-a} \epsilon^2 {b-\epsilon \choose a} = {b+2 \choose a+3 } + {b+1 \choose a+3}$$
\end{lem}
\begin{proof} The full detail of the proof has been omitted. It uses multiple times the fact that ${n+1 \choose r+1} = \sum_{j=r}^{n} {j \choose r}$

\end{proof} 
Using $b=(m+t-3)$, $a=m-2$, and replacing $\epsilon$ in the above lemma with $\epsilon + 1$,  it follows that $\Ext_R^{2m-3}(I^t/I^{t-1},S)$ has rank:  
$$ \dfrac{1}{m-1}{m+t-3 \choose m-2 }\left( {m+t-1 \choose m+1 }+{m+t -2 \choose m+1 } \right) $$ 
By a straightforward argument on long exact sequences of $Ext$ modules, $\Ext^{2m-3}_R(R/I^T,R)$ has rank 
\begin{align*} 
\ell \left( \Ext^{2m-3}_R(R/I^T,R) \right)&= \sum \ell \left(  \Ext^{2m-3}_R (I^{t-1}/I^t, R) \right) \\ 
 &= \sum\limits_{t= 1}^{T} \dfrac{1}{m-1}{m+t-3 \choose m-2 }\left( {m+t-1 \choose m+1 }+{m+t -2 \choose m+1 } \right) \\ 
&= \frac{1}{m-1} \sum\limits_{t= 2}^{T} {m+t-3 \choose m-2 }\left( {m+t-1 \choose m+1 }+{m+t -2 \choose m+1 } \right)
\end{align*}
\begin{lem} 
	\begin{multline*} 
	\frac{1}{m-1} \sum\limits_{t= 2}^{T} {m+t-3 \choose m-2 }\left( {m+t-1 \choose m+1 }+{m+t -2 \choose m+1 } \right) = \\  \frac{1}{m+1}{m+T-2 \choose m} {m+T -1 \choose m} 
	\end{multline*}  
\end{lem} 
\begin{proof}
	The proof is a standard induction argument. This equality is also verifiable using a computing software, such as Maple. \end{proof} 
We have just shown the following theorem. 
\myThm*

In \cite{DaoAndMontano}, they define 
$$ \epsilon^j = \frac{\ell(H^j_{\mf{m}}(R/I^t))}{n^{d}}. $$
This invariant is a generalization of $j$-multiplicity
$$ \epsilon = \frac{\ell (H^0_{\mf{m}}(R/I^t))} {n^{d}} $$ 
which, in turn, is a generalization of Hilbert-Samuel multiplicity, 
$$ e = (d)! \frac{\ell (H^0_{\mf{m}}(R/I^t))} {n^{d}} \text{ when } I \text{ is } \mf{m}\text{-primary.}$$
The presence of $d!$ in Hilbert-Samuel multiplicity indicates that perhaps the invariant $d! \epsilon^j$, a constant multiple of the one in \cite{DaoAndMontano}, may also be an enlightening one. 
\myCor*

	\pagebreak 

\bibliographystyle{alpha} 

\bibliography{main}

\end{document}